\documentclass[leqno,11pt, a4]{amsart}
\usepackage[utf8]{inputenc}
\usepackage{amsthm}
\usepackage{amsmath}
\usepackage{enumitem}
\usepackage{amsfonts}
\usepackage{amssymb}
\usepackage{hyperref}

\usepackage{graphicx}
\usepackage{pdfsync}
\usepackage{hyperref}
\usepackage{comment}
\usepackage{tikz}
\usepackage{hyperref}
\hypersetup{
    colorlinks=true,
    linkcolor=blue,
    filecolor=magenta,
    urlcolor=blue,
    }

\usepackage{graphicx}
\usepackage{pdfsync}
\usepackage{xcolor}
\usepackage{array}
\newtheorem{theorem}{Theorem}[section]
\newtheorem{definition}{Definition} [section]

\newtheorem{assumption}{Assumption}[section]
\newtheorem{corollary}{Corollary}[theorem]
\newtheorem{lemma}[theorem]{Lemma}
\newtheorem{proposition}[theorem]{Proposition}

\newcounter{yuppo}
\newcommand{\Keler} {K\"{a}hler }
\newcommand{\keler} {K\"{a}hler }
\newcommand{\kelerian} {K\"{a}hlerian }

\newcommand{\cd}{\cdot}
\renewcommand{\setminus}{-}

\renewcommand{\phi}{\varphi}

\newcommand{\lra}{\longrightarrow}
\newcommand{\C}{\mathbb{C}}
\newcommand{\R}{\mathbb{R}}
\newcommand{\Q}{\mathbb{Q}}

\newcommand{\Gl}{\operatorname{Gl}}

\newcommand{\Ad}{\operatorname{Ad}}

\newcommand{\enf}{\emph}

\newcommand{\desudt}[1] []      {\dfrac {\mathrm {d} #1 }{\mathrm {dt}}}
\newcommand{\desudtzero}        {\desudt \bigg \vert _{t=0} }
\newcommand{\zero}{\bigg \vert _{t=0} }

\newcommand{\liu}{\mathfrak{u}}

\newcommand{\lia}{\mathfrak{a}}
\newcommand{\lieh}{\mathfrak{h}}
\newcommand{\liek}{\mathfrak{k}}
\newcommand{\lier}{\mathfrak{r}}
\newcommand{\liel}{\mathfrak{l}}
\newcommand{\lieg}{\mathfrak{g}}

\newcommand{\liep}{\mathfrak{p}}
\newcommand{\lieq}{\mathfrak{q}}
\newcommand{\liez}{\mathfrak{z}}

\newcommand{\liet}{\mathfrak{t}}

\newcommand{\la}{\lambda}

\newcommand{\sx}{\langle} 
\newcommand{\xs}{\rangle}

\newcommand{\noparty}[1]{}

\newcommand{\scalo}{\sx \, , \, \xs}

\newcommand{\mup}{\mu_\liep}

\title{A Hilbert-Mumford Criterion for polystability for actions of real reductive Lie groups}
\author{Leonardo Biliotti}
\address{Dipartimento di Scienze Matematiche, Fisiche e Informatiche \\
          Universit\`a di Parma (Italy)}
\email{leonardo.biliotti@unipr.it}
\author{Oluwagbenga Joshua Windare}
\address{Dipartimento di Scienze Matematiche, Fisiche e Informatiche \\
          Universit\`a di Parma (Italy)}
\email{oluwagbengajoshua.windare@unipr.it}
\thanks{The first author was partially supported by PRIN  2017
   ``Real and Complex Manifolds: Topology, Geometry and holomorphic dynamics ''. Both authors were partially supported by ``National Group for Algebraic and Geometric Structures, and their Applications'' (GNSAGA - INDAM).}
\begin{document}
\maketitle
\begin{abstract}
\noindent
We presented a Hilbert-Mumford criterion for polystablility associated with an action of a real reductive Lie group $G$ on a real submanifold $X$ of a \Keler manifold $Z$. Suppose the action of a compact Lie group with Lie algebra $\liu$ extends holomorphically to an action of the complexified group $U^\C$ and that the $U$-action on $Z$ is Hamiltonian. If  $G\subset U^\C$ is compatible, there is a corresponding gradient map $\mu_\mathfrak{p}: X\to \mathfrak{p}$, where $\lieg = \liek \oplus \liep$ is a Cartan decomposition of the Lie algebra of $G$. Under some mild restrictions on the $G$-action on $X,$ we characterize which $G$-orbits in $X$ intersect $\mu_\liep^{-1}(0)$ in terms of the maximal weight function, which we viewed as a collection of maps defined on the boundary at infinity ($\partial_\infty G/K$) of the symmetric space $G/K$.
\end{abstract}
\maketitle
\section{Introduction}
\pagenumbering{arabic}
The classical Hilbert-Mumford criterion in Geometric Invariant Theory (in projective algebraic geometry) is an explicit numerical criterion for finding the stability of a point in terms of an invariant known as maximal weight function \cite{MUMFORD}. This criterion has been extended to the non-algebraic \kelerian settings using the theory of \keler quotients and a version of maximal weight function \cite{MUNDETC, MUNDETT, Teleman, BT, KLM}. For this setting, a \Keler manifold $(Z,\omega)$ with a holomorphic action of a complex reductive Lie group $U^\C$, where $U^\C$ is the complexification of a compact  Lie group $U$ with Lie algebra $\mathfrak{u}$ is considered. Assume $\omega$ is $U$-invariant and that there is a $U$-equivariant momentum map $\mu : Z \to \mathfrak{u}^*.$ By definition, for any $\xi \in \mathfrak{u}$ and $z\in Z,$ $d\mu^\xi = i_{\xi_Z}\omega,$ where $\mu^\xi(z) := \mu(z)( \xi )$ and $\xi_Z$ denotes the fundamental vector field induced on $Z$ by the action of $U,$ i.e.,
$$\xi_Z(z) := \frac{d}{dt}\zero \exp(t\xi)z
$$
(see, for example, \cite{Kirwan} for more details on the momentum map).

Our aim is to investigate a class of actions of real reductive Lie groups on real submanifolds of $Z$ using gradient map techniques. This setting was recently introduced in \cite{PG,heinzner-schwarz-stoetzel,heinzner-schuetzdeller}. More precisely, a subgroup $G$ of $U^\C$ is compatible if $G$ is closed and the map $K\times \mathfrak{p} \to G,$ $(k,\beta) \mapsto k\exp(\beta)$
is a diffeomorpism where $K := G\cap U$ and $\mathfrak{p} := \mathfrak{g}\cap \textbf{i}\mathfrak{u};$ $\mathfrak{g}$ is the Lie algebra of $G$.
The Lie algebra $\mathfrak{u}^\C$ of $U^\C$ is the direct sum $\mathfrak{u}\oplus \textbf{i}\mathfrak{u}.$ It follows that $G$ is compatible with the
Cartan decomposition of $U^\C = U\exp(\textbf{i}\mathfrak{u})$, $K$ is a maximal compact subgroup of $G$ with Lie algebra $\mathfrak{k}$ and that
$\mathfrak{g} = \mathfrak{k}\oplus \mathfrak{p}$. 
The inclusion $\textbf{i}\mathfrak{p}\hookrightarrow \mathfrak{u}$ induces by restriction, a $K$-equivariant map $\mu_{\textbf{i}\mathfrak{p}} : Z \to (\textbf{i}\mathfrak{p})^*.$ One can choose and fix an $\mathrm{Ad}(U^\C)$-invariant inner product $B$ of Euclidian type on the Lie algebra $\liu^\C$, see \cite[Section 3.2]{BT}, \cite[Definition 3.2.4]{LT1} and also \cite[Section 2.1]{He} for the analog in the algebraic GIT. Such an inner product will automatically induce a well-defined inner product on any maximal compact subgroup $U'$ of $U^\C$.

Let $\scalo$ denote the real part $B$. Then $\scalo$ is positive definite on $\textbf{i}\liu$, negative definite on $\liu$, $\langle \liu, \textbf{i} \liu \rangle=0$ and finally the multiplication by $\textbf{i}$ satisfies $\langle \textbf i \cdot ,\textbf{i} \cdot \rangle=-\scalo$.  In order to simplify the notation we replace consideration of $\mu_{\text{i}\mathfrak{p}}$ by that of $\mup:Z \lra \liep$, where
$$\mu_\mathfrak{p}^\beta (x):=\langle \mup(x),\beta \rangle:=\langle \textbf {i}\mu(x),\beta\rangle =-\langle \mu(x),-\textbf{i}\beta\rangle=\mu^{-\textbf{i} \beta}(x).$$
The map $\mup:Z \lra \liep$ is $K$-equivariant and grad$\, \mu_\mathfrak{p}^\beta = \beta_Z$ for any $\beta \in \liep$. Here the grad is computed with respect to the Riemannian metric induced by the \Keler structure. $\mu_\mathfrak{p}$ is called the $G$-gradient map associated with $\mu$.
For a $G$-stable locally closed real submanifold $X$ of $Z,$ we consider $\mu_\mathfrak{p}$ as a map $\mu_\mathfrak{p}: X\to \mathfrak{p}$ such that $\mathrm{grad}\, \mup=\beta_X$, where the gradient is now computed with respect to the induced Riemannian metric on $X$.

Different notions of stability of points in $X$ can be identified by taking into account the position of their $G$-orbits with respect to
 $\mu_\liep^{-1}(0).$ A point $x\in X$ is polystable if it's $G$-orbit intersects the level set $\mu_\liep^{-1}(0)$
 ($G\cdot x \cap \mu_\liep^{-1}(0) \neq \emptyset$). As pointed out in the introduction in \cite{MUNDETT} (see also \cite{Stability}), a set of polystable points plays a critical role in the construction of a good quotient of $X$ by the action of $G.$ The aim of this article is to answer the first part of question 1.1 in \cite{MUNDETT} for actions of real Lie groups on real submanifolds of a K\"ahler manifold, generalizing \cite{MUNDETT}.
Following \cite{MUNDETT}, we require a mild technical restriction to be satisfied; namely, the fundamental vector field induced by the action grows at most linearly with respect to the distance function from a given base point. More precisely,

 \begin{assumption}\label{assumption1} $X$ is connected, and
there exists a point $x_0\in X$ and a constant $C > 0$ such that for any $x\in X$ and any $\beta \in \liep,$
\begin{equation}\label{assumption}
    \parallel\beta_X(x)\parallel \leq C\parallel\beta\parallel(1 + d_X(x_0, x)),
\end{equation}

where $d_X$ denotes the geodesic distance between points of $X.$

\end{assumption}

$G$-action on $X$ when $X$ is compact, and a closed and a compatible subgroup $G \subset SL(n, \R)$-action on $\R^n$ satisfy the assumption.

Under this assumption, we construct a collection of maps $$\lambda_x: \partial_\infty (G/K) \to \R \cup \{\infty \}$$ called the maximal weight function. We then prove that a point $x\in X$ is polystable if and only if $\lambda_x \geq 0$ and for any $p\in \partial_\infty (G/K)$ such that $\lambda_x(p) = 0$ there exists $p'\in \partial_\infty (G/K)$ such that $p$ and $p'$ are connected in the sense of Definition \ref{deff} (Theorem \ref{main result}.)

The authors in \cite{Stability} gave a Hilbert-Mumford criterion for semistability and polystability under a certain completeness condition. However, the result of this article on polystability is much stronger as we did not assume any completeness condition.
\section{Compatible subgroups and parabolic subgroups}
\label{comp-subgrous}
Let $H$ be a Lie group with Lie algebra $\mathfrak{h}$ and $E, F \subset \mathfrak{h}$. Then, we set
 \begin{gather*}
   E^F :=     \{\eta\in E: [\eta, \xi ] =0, \forall \xi \in F\} \\
   H^F = \{g\in H: \Ad g (\xi ) = \xi, \forall \xi \in F\}.
 \end{gather*}
 If $F=\{\beta\}$ we write simply $E^\beta$ and $H^\beta$.

 Let $U$ be a compact Lie group and let $U^\C$ be the corresponding complex linear algebraic group \cite{chevalley}. The group $U^\C$ is reductive and is the universal complexification of $U$ in the sense of \cite{ho}. On the Lie algebra level, we have the Cartan decomposition $\liu^\C=\liu\oplus \textbf{i} \liu$ with a corresponding Cartan involution $\theta:\liu^\C \lra \liu^\C$ given by $\xi+\textbf{i}\nu \mapsto \xi-\textbf{i}\nu$. We also denote by $\theta$ the corresponding involution on $U^\C$. The real analytic map $F:U\times \textbf{i}\liu \lra U^\C$, $(u,\xi) \mapsto u\exp(\xi)$ is a diffeomorphism. We refer to the composition $U^\C=U\exp(\textbf{i}\liu)$ as the Cartan decomposition of $U^\C$.

 Let $G\subset U^\C$ be a closed real
 subgroup of $U^\C$. We say
 that $G$ is \enf{compatible} with the Cartan decomposition of $U^\C$ if $F (K \times \liep) = G$ where $K:=G\cap U$ and $\liep:= \lieg \cap \textbf{i}\liu$. The
 restriction of $F$ to $K\times \liep$ is then a diffeomorphism onto
 $G$. It follows that $K$ is a maximal compact subgroup of $G$ and
 that $\lieg = \liek \oplus \liep$.  Since $K$ is a retraction of $G$, it follows that $G$ has only finitely many connected components and $G=KG^o$, where $G^o$ denotes the connected component of $G$ containing $e$.  Since $U$ can be embedded in $\Gl(N,\C)$ for
 some $N$, and any such embedding induces a closed embedding of
 $U^\C$, any compatible subgroup is a closed linear group. By \cite[Proposition 1.59 p.57]{knapp-beyond},
 $\lieg$ is a real reductive Lie algebra, and so $\lieg =
 \liez(\lieg)\oplus [\lieg, \lieg]$, where $\liez(\lieg)=\{v\in \lieg:\, [v,\lieg]=0\}$ is the Lie algebra of the center of $G$. Denote by $G_{ss}$ the analytic
 subgroup tangent to $[\lieg, \lieg]$. Then $G_{ss}$ is closed and
 $G^o=Z(G^o)^o\cd G_{ss}$ \cite[p. 442]{knapp-beyond}, where $Z(G^o)^o$ denotes the connected component of the center of $G^o$ containing $e$.
 \begin{lemma}[\protect{\cite[Lemma 7]{LA}}]\label{lemcomp}$\ $ \begin{enumerate}
   \item \label {lemcomp1} If $G\subset U^\C$ is a compatible
     subgroup, and $H\subset G$ is closed and $\theta$-invariant,
     then $H$ is compatible if and only if $H$ has only finitely many connected components.
   \item \label {lemcomp2} If $G\subset U^\C$ is a connected
     compatible subgroup, then $G_{ss}$ is compatible.
     \item \label{lemcomp3} If $G\subset U^\C$ is a compatible
       subgroup and $E\subset \liep$ is any subset, then $G^E$ is
       compatible. Indeed, $G^E=K^E\exp(\liep^E)$, where $K^E=K\cap G^E$ and $\liep^E=\{x\in \liep:\, [x,E]=0\}$.
   \end{enumerate}
 \end{lemma}
If $\beta \in \liep$ the endomorphism $\mathrm{ad}(\beta) \in\mathrm{End}(\lieg)$ is diagonalizable over $\R$. Denote by $V_\lambda (\mathrm{ad}(\beta))$ the eigenspace of $\mathrm{ad}(\beta)$ corresponding to the eigenvalue $\lambda$. We define,
\begin{align*}
  G^{\beta+} &:=\{g \in G : \lim_{t\to - \infty} \exp (t\beta)\, g
  \exp (-t\beta) \text { exists} \}\\
  R^{\beta+} &:=\{g \in G : \lim_{t\to - \infty} \exp (t\beta)\,
  g \exp (-t\beta) =e \}\\
    \lier^{\beta+}: = \bigoplus_{\la > 0} V_\la (\mathrm{ad}( \beta)) \\
  G^{\beta-} &:\{g\in G:\, \lim_{t\to + \infty} \exp (t\beta)\, g \exp (-t\beta) \text { exists}\}\\
   R^{\beta-} &:=\{g \in G : \lim_{t\to + \infty} \exp (t\beta)\, g \exp(-t\beta)=e \}\\
  \lier^{\beta-}: = \bigoplus_{\la < 0} V_\la (\mathrm{ad}(\beta)).
\end{align*}
Note that $G^{\beta-}=G^{(-\beta)+}=\theta(G^{\beta+})$, $\theta(R^{\beta+})=R^{\beta-}$ and $G^{\beta}=G^{\beta+}\cap G^{\beta-}$.
The group $G^{\beta+}$, respectively $G^{\beta-}$, is a parabolic subgroup of $G$ with unipotent radical $R^{\beta+}$, respectively $R^{\beta-}$ and Levi factor $G^{\beta}$.
$R^{\beta+}$ is connected with Lie algebra $\lier^{\beta+}$. Hence $R^{\beta-}$ is connected with Lie algebra $\lier^{\beta-}$.
 The parabolic subgroup $G^{\beta+}$ is the semidirect product of $G^\beta$ with $R^{\beta+}$ and we have the projection $\pi^{\beta+} : G^{\beta+} \lra G^\beta$,
 $
 \pi^{\beta+}(g)=\lim_{t\to -\infty} exp(t\beta)g\exp(-t\beta).
 $
Analogously, $G^{\beta-}$ is the semidirect product of $G^{\beta}$ with $R^{\beta-}$ and we have the projection $\pi^{\beta-}:G^{\beta-} \lra G^\beta$,  $\pi^{\beta-}(g)=\lim_{t\to +\infty} exp(t\beta)g\exp(-t\beta)$. In particular, $ \lieg^{\beta+}= \lieg^\beta \oplus \lier^{\beta+}$, respectively  $ \lieg^{\beta-}= \lieg^\beta \oplus \lier^{\beta-}$.
\begin{proposition}\label{parabolic-decomposition}
For any $\beta \in \liep$, we have $G=KG^{\beta+}$.
\end{proposition}
\begin{proof}
If $G$ is connected, the result is well-known, see for instance \cite[Lemma 9]{LA} and \cite[Lemma 4.1]{heinzner-schwarz-stoetzel}. Since $G=KG^o$, it follows that
$
G=KG^o=K(G^{o})^{\beta+}=KG^{\beta+},
$
concluding the proof.
\end{proof}
\section{Symmetric space}
Let $U^\C$ be a complex reductive Lie group and let $G\subset U^\C$ be a closed compatible subgroup. Then $G = K\exp(\mathfrak{p}),$ where $K := G\cap U$ is a maximal compact subgroup of $G$ and $\mathfrak{p} := \mathfrak{g}\cap \text{i}\mathfrak{u};$ $\mathfrak{g}$ is the Lie algebra of $G.$ Let $M = G/K$ and let $\sx , \xs$ be the real part of an $\Ad(U^\C)$-invariant inner product $B$ of Euclidean type on the Lie algebra $\liu^\C.$  Equip $G$ with the unique left-invariant Riemannian metric which agrees with the scalar product $\sx \xi_1 + \beta_1, \xi_2 + \beta_2\xs:= -\sx\xi_1, \xi_2\xs + \sx\beta_1, \beta_2\xs$ where $\xi_1,\xi_2\in \liek$ and $\beta_1,\beta_2\in \liep$. We recall that $\langle \cdot,\cdot \rangle$ is negatively defined on $\liek$. This metric is $\mathrm{Ad}(K)$-invariant, and so, it induces a $G$-invariant Riemannian metric of nonpositive curvature on $M.$ $M$ is a symmetric space of non-compact type \cite{borel-ji-libro}. Let $\pi: G \rightarrow M$ be the projection onto the right cosets of $G$. $G$ acts isometrically on $M$ from left by
$$L_g :  M \to M; \quad L_g(hK) := ghK, \quad g,h \in G.$$ A geodesic $\gamma$ in $M$ is given by $\gamma = g\exp(t\beta)K,$ where  $g\in G$ and $\beta \in \liep.$ For $\beta \in \liep$, we set $\gamma^\beta(t) = \exp(t\beta)K$ and $o := K \in M.$

Since $M$ is a Hadamard manifold there is a natural notion of a boundary at infinity which can be described using geodesics. We refer the reader to \cite{borel-ji-libro} for more details. Two unit speed geodesics $\gamma, \gamma': \R \to M$ are equivalent, denoted by $\gamma \sim \gamma'$, if $\sup_{t>0}d(\gamma(t), \gamma'(t)) < +\infty.$

\begin{definition}
The \textit{Tits boundary} of $M$ denoted by $\partial_\infty M$ is the set of equivalence classes of unit speed geodesics in $M.$
\end{definition}

The map that sends $\beta \in \liep$ to the tangent vector $\dot{\gamma}^\beta(0)$ produces an isomorphism $\liep \cong T_oM.$ Since any geodesic ray in $M$ is equivalent to a unique ray starting from $o,$ the map
\begin{align*}
e: &S(\liep) \to \partial_\infty M;\\
&e(\beta) := [\gamma^\beta]
\end{align*}
where $S(\liep) := \{\beta \in \liep : |\beta| = 1\}$ is the unit sphere in $\liep,$ is a bijection. The \textit{sphere topology} is the topology on $\partial_\infty M$ such that $e$ is a homomorphism. Since $G$ acts by isometries on $M,$ then for every unit speed geodesic $\gamma,$ $g\gamma$ is also a unit speed geodesic for any $g\in G.$ Moreover, if $\gamma \sim \gamma'$ then $g\gamma \sim g\gamma'.$ There is a $G$-action on $\partial_\infty M$ given by:
$$g\cdot [\gamma] = [g\cdot \gamma]$$ and this action also induces a $G$-action on $S(\liep)$ given by $$g\cdot \beta := e^{-1}(g\cdot e(\beta)) = e^{-1}[g\cdot \gamma^\beta].$$ We point out that the $K$-action to $\partial_\infty M$ induces the adjoint action of $K$ on $S(\liep).$

Let $H$ be a compatible subgroup of $G,$ i.e $H := L \exp(\lieq),$ where $L := H \cap K$ and $\lieq = \lieh \cap \liep$, where $\lieh$ is the Lie algebra of $H.$ It follows that $H$ is a real reductive subgroup of $G.$ The Cartan involution of $G$ induces a Cartan involution of $H,$ $L$ is a maximal compact subgroup of $H,$ and $\lieh = \liel \oplus \lieh.$ There are totally geodesic inclusions $M' := H/L \hookrightarrow M$ and $\partial M' \hookrightarrow \partial M.$

\subsection{The Kempf-Ness Function}\label{Kempf-Ness Function}
Given $G$ a real reductive group which acts smoothly on $Z;$ $G = K\text{exp}(\mathfrak{p}),$ where $K$ is a maximal compact subgroup of $G.$ Let $X$ be a $G$-invariant locally closed submanifold of $Z.$ As Mundet pointed out in \cite{MUNDET}, there exists a function $\Phi : X \times G \rightarrow \mathbb{R},$ such that
$$
\langle \mu_\mathfrak{p}(x), \xi\rangle = \desudtzero \Phi (x, \exp(t\xi)), \qquad \xi \in \mathfrak{p},
$$ and satisfying the following conditions:

\begin{enumerate}
    \item For any $x\in X,$ the function $\Phi (x, .)$ is smooth on $G.$
    \item The function $\Phi(x, .)$ is left-invariant with respect to $K,$ i.e., $\Phi(x, kg) = \Phi(x, g).$
    \item For any $x\in X,$ $v\in \mathfrak{p}$ and $t\in \mathbb{R};$

    $$\frac{d^2}{dt^2}\Phi (x, \exp(tv)) \geq 0.$$
    Moreover: $$\frac{d^2}{dt^2}\Phi (x, \exp(tv)) = 0$$ if and only if $\exp(\mathbb{R}v)\subset G_x.$

    \item For any $x\in X,$ and any $g, h \in G;$
    $$\Phi(x, hg) = \Phi(x, g) + \Phi(gx, h).$$ This equation is called the cocycle condition. Finally, using the cocycle condition, we have
\[
\desudt \Phi(x,\exp(t\beta))=\langle \mup(\exp(t\xi)x),\beta\rangle.
\]
\end{enumerate}
The function $\Phi: X \times G \rightarrow \mathbb{R}$ is called the Kempf-Ness function for $(X, G, K)$. It is just the restriction of the classical Kempf-Ness function $Z\times U^\C \lra \R$ considered in \cite{MUNDET, Teleman} to $X\times G$  \cite{LM}. Moreover, if $H\subset G$ is compatile and $Y\subset X$ is a $H$-stable submanifold of $X$, then the restriction $\Phi_{\vert{Y\times H}}$ is the Kempf-Ness function of the $H$-gradient map on $Y$.

Let $x\in X$. By property $(b),$ i.e. $\Phi(x, kg) = \Phi(x, g)$, the function $\Phi_x : G \to \mathbb{R}$ given by $\Phi_x(g) := \Phi(x, g^{-1})$  descends to a function on $M$ which we denote by the same symbol. That is
$$\Phi_x : M \lra \mathbb{R}; \qquad \Phi_x (gK):=\Phi(x,g^{-1}).$$

The cocycle condition $(d)$ can be rewritten as
\begin{equation}\label{cocycle}
    \Phi_x(ghK) = \Phi_{g^{-1}x}(hK) + \Phi_x(gK),
\end{equation} and it is equivalent to $L_g^{*}\Phi_x = \Phi_{g^{-1}x} + \Phi_{g^{-1}x}(gK),$ where $L_g$ denotes the action of $G$ on $X$ given above.

Note that
$$-(d\Phi_x)_o(\dot{\gamma}^\beta(0)) = \desudtzero \Phi_x(\exp(-t\beta)K) = \desudtzero \Phi(x, \exp(t\beta)) = \langle \mu_\mathfrak{p}(x), \beta\rangle.$$

\begin{lemma}\label{ConvexKempfness}
Let $x\in X$ and let $\Phi_x : M \to \mathbb{R}.$ Suppose $\gamma(t) = g\exp(t\beta)K$ for $\beta \in \liep$ is a geodesic in $M,$ then
$\Phi_x \circ \gamma$ is  convex and so, $$\lim_{t\to \infty}\frac{d}{dt}(\Phi_x\circ \gamma) =  \lim_{t\to \infty}\frac{\Phi_x\circ \gamma}{t}$$.
\end{lemma}

\begin{proof}
That $\Phi_x$ is a convex function on $M$ follows from \cite[Lemma 2.19]{bgs}. Let $f(t) = (\Phi_x\circ \gamma)(t).$ Since $f$ is convex,
$$\frac{f(s)}{s}\leq f'(s) \leq \frac{f(t) - f(s)}{t-s} \quad 0 < s < t.$$ Furthermore, the two quantities are increasing in $s,$ while the third in $t.$ Hence,
$$\lim_{s\to \infty}\frac{f(s)}{s}\leq \lim_{s \to \infty} f'(s) \leq \lim_{t\to \infty} \frac{f(t) - f(s)}{t-s}.$$ Since the last limit is the same with the first limit, we have $\lim_{t \to \infty} f'(t) = \lim_{t\to \infty}\frac{f(t)}{t}$ and the result follows.
\end{proof}

\section{Stability and Maximal Weight Function}\label{analytic}
Let $(Z,\omega)$ be a \Keler manifold and $U^\C$ acts holomorphically on $Z$ with a momentum map $\mu : Z \to \mathfrak{u}.$ Let $G\subset U^\C$ be a closed compatible subgroup. $G = K\exp(\mathfrak{p}),$ where $K := G\cap U$ is a maximal compact subgroup of $G$ and $\mathfrak{p} := \mathfrak{g}\cap \text{i}\mathfrak{u};$ $\mathfrak{g}$ is the Lie algebra of $G.$

Suppose $X\subset Z$ is a $G$-stable locally closed connected real submanifold of $Z$ with the gradient map $\mu_\mathfrak{p} : X\to \mathfrak{p}.$ Let $d_X$ denote the geodesic distance between points of $X.$ We recall that by $G_x$ and $K_x$, we denote the stabilizer subgroup of $x\in X$ with respect to the $G$-action and the $K$-action respectively and by $\lieg_x$ and $\liek_x$ their respective Lie algebras.
\begin{definition}
Let $x\in X.$ Then:
\begin{enumerate}
\item $x$ is stable if $G\cdot x \cap \mu_\liep^{-1}(0) \neq \emptyset$ and $\lieg_x$ is conjugate to a Lie subalgebra of $\liek.$
\item $x$ is polystable if $G\cdot x \cap \mu_\liep^{-1}(0) \neq \emptyset.$
\item $x$ is semistable if $\overline{G\cdot x} \cap \mu_\liep^{-1}(0) \neq \emptyset.$
\end{enumerate}
\end{definition}
We denote by $X^s_{\mup}$, $X^{ss}_{\mup}$, $X^{ps}_{\mup}$ the set of stable, respectively semistable, polystable, points.
It follows directly from the definitions above that the conditions are $G$-invariant in the sense that if a point satisfies one of the conditions, then every point in its orbit satisfies the same condition, and for stability, recall that $\lieg_{gx} = \text{Ad}(g)(\lieg_x).$

The following well-known result establishes a relation between the Kempf-Ness function and the polystability condition. A proof is given in \cite{Stability}.
\begin{proposition}\label{polystable-I}
Let $x\in X$ and let $g\in G$. The following conditions are equivalent:
\begin{enumerate}
\item $\mup(gx)=0$.
\item $g$ is a critical point of $\Phi(x,\cdot)$.
\item $g^{-1}K$ is a critical point of $\Phi_x$.
\end{enumerate}
\end{proposition}
\begin{proposition}\label{propp}
Let $x\in X$.
\begin{itemize}
\item If $x$ is polystable, then $G_x$ is compatible.
\item If $x$  is stable, then $G_x$ is compact.
\end{itemize}
\end{proposition}

\subsubsection{Maximal Weight Function}

In this section, we introduce the maximal weight function associated with an element $x\in X.$

For any $t\in \mathbb{R},$ define $\la(x,\beta,t) = \langle\mu_\mathfrak{p}(\exp(t\beta)x), \beta\rangle.$
\begin{equation}\label{form1}
\la(x,\beta,t) = \langle\mu_\mathfrak{p}(\exp(t\beta)x), \beta\rangle = \frac{d}{dt}\Phi(x, \exp(t\beta)),
\end{equation}
 where $\Phi: X\times G\to \mathbb{R}$ is the Kempf-Ness function. By the properties of the Kempf-Ness function,
$$
\frac{d}{dt}\la(x,\beta,t) = \frac{d^2}{dt^2}\Phi(x, \exp(t\beta)) \geq 0.
$$
This means that $\la(x,\beta,t)$ is a non decreasing function as a function of $t.$

The maximal weight of $x\in X$ in the direction of $\beta \in \mathfrak{p}$ is defined in \cite{Stability} as the numerical value
$$
\la(x,\beta) = \lim_{t\to \infty}\la(x,\beta,t)=\lim_{t\to +\infty} \langle\mu_\mathfrak{p}(\exp(t\beta)x), \beta\rangle \in \mathbb{R}\cup\{\infty\}.
$$
Note that $$\frac{d}{dt}\la(x,\beta,t) = \parallel\beta_X(\exp(t\beta) x)\parallel^2,
$$ and so,
\begin{equation}\label{Equu}
\la(x,\beta,t) =  \langle\mu_\mathfrak{p}(x), \beta\rangle + \int_0^t\parallel\beta_X(\exp(s\beta) x)\parallel^2 \mathrm{ds}.
\end{equation}
\begin{lemma}\label{le1}
Let $\beta,\beta'\in \liep$ and let $x\in X$. If $\beta \in \lieg_x$ and $[\beta,\beta']=0$, then
\[
\lim_{t\to +\infty} \frac{d}{dt}\Phi(x, \exp(t(\beta+\beta'))=\lim_{t\to +\infty} \frac{d}{dt}\Phi(x, \exp(t\beta))+\frac{d}{dt}\Phi(x, \exp(t\beta')),
\]
\end{lemma}
\begin{proof}
By the cocycle condition, keeping in mind that $[\beta,\beta']=0$, we have
\[
\Phi(x, \exp(t(\beta+\beta'))=\Phi(x,\exp(t\beta'))+\Phi(\exp(t\beta)x,\exp(t\beta'))=
\Phi(x,\exp(t\beta'))+\Phi(x,\exp(t\beta')),
\]
and so the result follows.
\end{proof}

\begin{lemma}\label{mylemma}
    Let $g\in G$ and $\beta \in \liep.$ If $\frac{\Phi_x(\exp(t\beta)g)}{t}$ is bounded uniformly on $t,$ then $$\lim_{t \to \infty}\frac{d_X(\exp(t\beta)gx, x)}{t^{\frac{1}{2}}} = 0.$$
\end{lemma}
\begin{proof} By the property (d) of the Kempf-Ness function and noting that $\lim_{t \to \infty}\frac{\Phi_x(g)}{t} = 0$ since $\Phi_x(g)$ is independent of $t$, we have
\begin{align*}
    \lim_{t \to \infty}\frac{\Phi_x(\exp(t\beta)g)}{t} &= \lim_{t \to \infty}\frac{\Phi_x(g) + \Phi_{gx}(\exp(t\beta))}{t}\\
    &=\lim_{t\to \infty}\frac{\Phi_{gx}(\exp(t\beta))}{t} \\
    &= \lambda(gx, \beta).
\end{align*} If $\frac{\Phi_x(\exp(t\beta)g)}{t}$ is bounded uniformly on $t,$ then by (\ref{Equu}), $\int_0^\infty\parallel \beta_X(\exp(s\beta)gx)\parallel^2 ds < \infty.$ On the other hand,
$$d_X(\exp(t\beta)gx, gx) \leq \int_0^t \parallel \beta_X(\exp(s\beta)gx)\parallel ds.$$ Let $f(s) = \parallel \beta_X(\exp(s\beta)gx)\parallel,$ applying \cite[Lemma 3.2]{MUNDETT} to $f(s)$ implies that $$\lim_{t\to \infty}\frac{d_X(\exp(t\beta)gx, gx)}{t^{\frac{1}{2}}} = 0,$$ and by triangle inequality, $$\lim_{t\to \infty}\frac{d_X(\exp(t\beta)gx, x)}{t^{\frac{1}{2}}} = 0.$$
\end{proof}

For $x\in X,$ let the function $\lambda_x$ be given as

\begin{equation}\label{mmfunction}
\lambda_x : \partial_\infty M \to \R; \quad \lambda_x([\gamma]) := \lim_{t\to \infty}\frac{d}{dt}\Phi_x(\gamma(t)) = \lim_{t\to \infty}\frac{\Phi_x\circ \gamma}{t},
\end{equation} where the last equality follows from Lemma \ref{ConvexKempfness}. We need to show that the function $\lambda_x$ is well defined using the technical condition in assumption \ref{assumption1}. We first derive a quadratic bound on the gradient map.

\begin{lemma}\label{quadratic bound}
    Let $x_0\in X$ be the point satisfying (\ref{assumption}). There exists a constant $C_1$ such that for any $x\in X$ we have $$\parallel\mu_\liep(x)\parallel \leq C_1(1 + d_X(x_0, x)^2).$$
\end{lemma}

\begin{proof}
    For any $x\in X$, let $d = d_X(x_0, x).$ By (\ref{assumption}), there exists a constant $C > 0$ such that for any $x\in X$ and any $\beta \in \liep,$ $\parallel\beta_X(x)\parallel \leq C \parallel\beta\parallel(1+d).$ Let $\alpha : [0, 2d] \to X$ be a smooth curve parametrized by the arc length and such that $\alpha(0) = x_0$ and $\alpha(2d) = x.$ For any $\beta \in \liep,$ let $f(t) := \langle \mu_\liep(\alpha(t)), \beta\rangle.$ Then by the defining properties of gradient map and since $\parallel\Dot{\alpha}\parallel = 1$
\begin{align*}
    f'(t) &= \langle d\mu_\liep(\alpha(t))\Dot{\alpha}, \beta\rangle\\
    &= \langle \beta_X(\alpha(t)), \Dot{\alpha}\rangle\\
    &\leq \parallel\beta_X(\alpha(t))\parallel\\
    & \leq C \parallel\beta\parallel(1+d_X(x_0, \alpha(t)).
\end{align*}
Hence,
\begin{align*}
    f(2d) - f(0) &= \int_0^{2d} f'(t)dt \leq \int_0^{2d} C \parallel\beta\parallel(1+d_X(x_0, \alpha(t))dt\\
    & \leq \int_0^{2d}  C \parallel\beta\parallel(1+t)dt\\
    & = 2C\parallel\beta\parallel(d + d^2)\\
    &\leq C'\parallel\beta\parallel(1 + d^2),
\end{align*} where the constant $C'$ can be chosen independent of $x.$ Hence the result follows.
\end{proof}

\begin{proposition}
The function $\lambda_x : \partial_\infty M \to \R$ defined by (\ref{mmfunction}) is well defined.
\end{proposition}

\begin{proof}
    The function $\lambda_x$ is well defined means that if $\gamma_1$ and $\gamma_2$ are geodesic rays such that $\gamma_1 \sim \gamma_2,$ then $\lambda_x(\gamma_1) = \lambda_x(\gamma_2).$ Let $\gamma_i(t) = [g_i\exp(t\beta_i)]$ for $i = 1,2.$  Suppose $\lambda(\gamma_1) < \infty.$

    $$
    \lim_{t\to \infty}\frac{\Phi_x\circ \gamma_1}{t} = \lim_{t\to \infty}\frac{\Phi_x(g_1\exp(t\beta_1)K)}{t} = \lim_{t\to \infty}\frac{\Phi_x(\exp(-t\beta_1)g^{-1}_1)}{t} < \infty.
    $$

    Lemma \ref{mylemma} implies that $$\lim_{t\to \infty}\frac{d_X(\exp(-t\beta_1)g^{-1}_1x, x)}{t^{\frac{1}{2}}} = 0.$$ Since $\gamma_1 \sim \gamma_2,$ then $d(\gamma_1(t), \gamma_2(t)) \leq Q$ uniformly on $t.$ Hence, for any $t$ one can consider a smooth map $\phi_t : [1, 2] \to G$ such that $\phi_t(i) = \exp(-t\beta_i)g_i^{-1}$ for $i = 1,2$ and
    \begin{equation}\label{smooth map}
    \int_1^2 \parallel\phi'_t(r) \phi_t(r)^{-1}\parallel dr \leq Q.
    \end{equation} Note that $Q$ is independent of $t.$

    We claim that there exists some constant $C'$ (depending on $x$) such that, for any $t$ and any $r\in [1,2],$
$$
d_X(\exp(-t\beta_1)g_1^{-1}x, \phi_t(r)x) \leq C'(1 + d_X(\exp(-t\beta_1)g_1^{-1}x, x)).
$$ Indeed, fix $t$ and let $c(u) = \phi_t(u)x$ and $\xi(u) = \parallel \phi'_t(u) \phi_t(u)^{-1}\parallel.$ Define $f(s) = d_X (c(1), c(s)).$ By assumption \ref{assumption1},

\begin{align*}
\parallel f'(s) \parallel \leq \parallel c'(s) \parallel &= \parallel (\phi'_t(s) \phi_t(s)^{-1})_X(c(s)) \parallel\\
&\leq C \parallel \phi'_t(s) \phi_t(s)^{-1} \parallel (1 + d_X(x_0, c(s)))\\
&\leq C\xi(s)(1 + d_X(x_0, x) + d_X(x, c(1)) + d_X(c(1), c(s)))\\
& = C\xi(s)(1 + d_X(x_0, x) + d_X(x, c(1)) + f(s)).
\end{align*}

If $g(s) = 1 + d_X(x_0, x) + d_X(x, c(1)) + f(s),$ then $g'(s) = f'(s)$ and $\parallel g'(s)/g(s)\parallel \leq C\xi(s).$ Integrating over $s\in [1, r],$ implies that there is a uniform bound $g(r) \leq CQg(1).$ Note that $d_X(x, c(1)) = d_X(x,  \exp(-t\beta_1)g_1^{-1}x).$ The claim follows.

Using triangular inequality and by the claim, we have
\begin{align*}
d_X(\phi_t(r)x, x) &\leq d_X(\phi_t(r)x, \exp(-t\beta_1)g_1^{-1}x) + d_X(\exp(-t\beta_1)g_1^{-1}x, x)\\
 &= C'(1 + d_X(\exp(-t\beta_1)g_1^{-1}x, x)) + d_X(\exp(-t\beta_1)g_1^{-1}x, x).
\end{align*} So that

$$
\lim_{t\to \infty}\frac{d_X(\phi_t(r)x, x)}{t^{\frac{1}{2}}} = 0
$$ and by triangular inequality again,

\begin{equation}\label{distance}
\lim_{t\to \infty}\frac{d_X(\phi_t(r)x, x_0)}{t^{\frac{1}{2}}} = 0.
\end{equation}

By Lemma \ref{quadratic bound}, $$\parallel\mu_\liep(\phi_t(r)x)\parallel \leq C_1(1 + d_X(x_0, \phi_t(r)x)^2),$$ and by (\ref{distance}),
$$
\lim_{t\to 0}\frac{\parallel\mu_\liep(\phi_t(r)x)\parallel}{t} = 0.
$$

This implies that
\begin{align*}
    \lim_{t\to \infty}\frac{\parallel \Phi_x(\gamma_1) - \Phi_x(\gamma_2)\parallel}{t} = 0
\end{align*} Therefore, $\lambda_x(\gamma_1) = \lambda_x(\gamma_2).$ If $\lambda_x(\gamma_1) = \infty,$ then $\lambda_x(\gamma_2) = \infty.$ Otherwise, reversing the roles of $\gamma_1$ and $\gamma_2$ in the arguments above shows that $\lambda_x(\gamma_1) = \lambda_x(\gamma_2) < \infty,$ which is a contradiction. Concluding the proof.

\end{proof}

\begin{definition}
The maximal weight function of $x\in X$ is the well defined map $\lambda_x : \partial_\infty M \to \R \cup \{\infty\}$ given as
$$
\la_x(p) := \la_x(\gamma) = \lim_{t\to \infty}\frac{d}{dt}\Phi_x(\gamma(t))
$$ for any $p\in \partial_\infty M,$ where $\gamma$ is any geodesic ray representing $p.$
\end{definition}
It follows from the definition that for any $\beta \in S(\liep)$, keeping in mind formula $(\ref{form1})$, we have
\begin{equation}\label{formula}
\lambda_x(e(\beta)) =  \lim_{t\to \infty}\frac{d}{dt}\Phi_x(\exp(t\beta)K) = \lim_{t\to \infty}\frac{d}{dt}\Phi(x, \exp(-t\beta))=\lim_{t\to +\infty} \langle \mup(\exp(t\beta)x,-\beta \rangle.
\end{equation}

\begin{lemma}\label{le3}
Let $\beta \in \liep\setminus\{0\}$ and let $v=\frac{\beta}{\parallel \beta \parallel}$. Then
\[
\lambda_x (e(v))= \frac{1}{\parallel \beta \parallel} \lim_{t\to +\infty} \frac{d}{dt}\Phi(x, \exp(-t\beta)).
\]
\end{lemma}
\begin{proof}
$\Phi_x (\exp(tv))=\Phi_x (\exp \big(\frac{t}{\parallel \beta \parallel} \beta \big))$. Then
\[
\lambda_x (e(\beta))=\frac{1}{\parallel \beta \parallel} \lim_{t\to +\infty} \frac{d}{dt}\Phi(x, \exp(-t\beta)).
\]
\end{proof}

\begin{lemma}
For any $x\in X,$ any $g\in G$ and any $p\in \partial_\infty M,$
\begin{equation}\label{Ginvariant}
    \la_{g^{-1}x}(p) = \la_x(g\cdot p).
\end{equation}
\end{lemma}

\begin{proof}
Let $p = [\gamma]$ for some geodesic $\gamma \in M.$ Then $g\cdot p = [g\circ \gamma].$ By the cocycle condition,

\begin{align*}
    \lambda_{g^{-1}x}(p) &= \lim_{t\to \infty}\frac{d}{dt} (\Phi_{g^{-1}x}(\gamma(t))\\
    &= \lim_{t\to \infty}\frac{d}{dt} (\Phi_{x}(g \cdot \gamma(t)) - \Phi_{x}(gK))\\
    &= \lim_{t\to \infty}\frac{d}{dt} (\Phi_{x}(g \cdot \gamma(t))) = \lambda_x(g\cdot p).
\end{align*}
\end{proof}
The following lemma will be needed. A proof is given in \cite[Lemma 3.5, p. 92]{Stability}, see also \cite{Teleman}.
\begin{lemma}\label{le2}
Let $V$ be a subspace of $\liep$. The following are equivalent for a point $x\in X$:
\begin{enumerate}
\item the map $\Phi(x,\cdot)$ is linearly properly on $V$, i.e., there exists positive constants $C_1$ and $C_2$ such that
    $$
    \parallel v \parallel \leq C_1 \Phi(x,\exp(v))+C_2,\, \forall v\in V.
    $$
\item $\lambda(x,\beta)>0$ for every $\beta\in S(V)$.
\end{enumerate}
\end{lemma}

The following theorem gives a numerical criterion for stable points in terms of maximal weights. The proof is the same as the proof of \cite[Theorem 3.7]{Stability}.
\begin{theorem}\label{stable}
Let $x\in X$. Then $x$ is stable if and only if $\lambda_x > 0$ on $\partial_\infty M.$
\end{theorem}
\begin{definition}\label{deff}
    We say that $p, q \in \partial_\infty M$ are connected if there exists a geodesic $\alpha$ in $X$ such that $p = \alpha(\infty)$ and $q = \alpha(-\infty).$
\end{definition}
For any $x\in X,$ as in \cite{MUNDETT}, see also \cite{bgs}, let $Z(x) := \{p\in \partial_\infty M : \lambda_x(p) = 0\}.$

\begin{lemma}\label{compatible-Stabilizer}
  Let $x\in X$ be such that $\mu_\liep(x) = 0,$ then $\lieg_x = \liek_x \oplus \liep_x$ and $Z(x) = e(S(\liep_x)) = \partial_\infty G_x/K_x.$
\end{lemma}

\begin{proof}
    By Proposition \ref{propp} if $\mu_\liep(x) = 0,$ $G_x$ is a compatible subgroup of $G.$ Hence, $\lieg_x = \liek_x \oplus \liep_x.$
To prove the second assertion, let $\beta \in S(\liep).$ Suppose $e(\beta) \in Z(x).$ This means that $\lambda_x(e(\beta)) = 0,$ then the convex function $f(t) := \Phi_x(\exp(t\beta)K)$ satisfies
$$
f'(\infty) = \lim_{t\to \infty}\frac{d}{dt}\Phi_x(\exp(t\beta)K) = \lambda_x(e(\beta)) = 0
$$
and
$$
f'(0) = \frac{d}{dt}|_{t= 0}\Phi_x(\exp(t\beta)K) = \langle \mu_\liep(x), -\beta \rangle = 0.
$$
These imply that $f$ is constant for all $t> 0,$ and by the condition (c) of Kempf-Ness function, $\exp(\R\beta)\subset G_x.$ Since $G_x$ is compatible, $\beta \in S(\liep_x).$ Conversely, if $\beta \in S(\liep_x),$ then $f$ is linear. Moreover, $f'(0) = 0.$ Therefore, $f \equiv 0$ and $e(\beta) \in Z(x).$
\end{proof}
Let $x\in X$ and $\beta \in S(\liep).$ Since $\liep \subset i\liu,$ then $i\beta \in \liu.$ We define the torus $T_\beta$ given as
$$
T_\beta := \overline{\{\exp(ti\beta) : t\in \R\}} \subseteq U^o,
$$
where $U^o$ denotes the connected component of $U$ containing the identity.
\begin{lemma}\label{torus}
    Let $g\in G.$ Then $\dim T_\beta = \dim T_{g\cdot \beta}.$
\end{lemma}
\begin{proof}
 It is well-known that $G^{\beta+}$ fixes $e(\beta)$, see for instance \cite[Proposition 2.17.3, p.102]{Eberlein1997GeometryON}. Then for any $g\in G$, write $g = kh,$ where $k\in K$ and $h\in G^{\beta+},$ so that $g\cdot \beta=kh\cdot \beta=k\cdot \beta=\mathrm{Ad}(k)(\beta)$. Hence
    \begin{align*}
        T_{g\cdot \beta} &= \overline{\{\exp(it\mathrm{Ad}(k)\beta) : t\in \R\}}\\
        &= \overline{\{k\exp(it\beta)k^{-1} : t\in \R\}}\\
        &= kT_\beta k^{-1}.
    \end{align*} Hence, $\dim T_\beta = \dim T_{g\cdot \beta}.$
\end{proof}

\begin{lemma}\label{limit}
    Let $x\in X$ and $p, p'\in Z(x)$ be connected. Then there exists $g\in G$ and $\xi\in S(p)$ such that $\xi\in \liep_y$, where $y = gx.$
\end{lemma}

\begin{proof}
    Since $p, p'\in Z(x)$ are connected, then there exists geodesic $\alpha \in M$ such that $\alpha(+\infty) = p\in Z(x)$ and $\alpha(-\infty) = p' \in Z(x).$ Assume $\alpha(t) = g\exp(t\xi)K,$ $g\in G.$ Then $p = g\cdot e(\xi)$ and $p' = g\cdot e(-\xi).$ By (\ref{Ginvariant}),

$$\lambda_{g^{-1}x}(e(\xi)) = \lambda_x(g\cdot e(\xi)) = \lambda_x(p) = 0$$ and
$$
\lambda_{g^{-1}x}(e(-\xi)) = \lambda_x(g\cdot e(-\xi)) = \lambda_x(p') = 0.
$$
Let $y = g^{-1}x.$ This means that the convex function $t \mapsto \Phi_y(\exp(t\xi)K)$ has zero derivatives at both $+\infty$ and $-\infty,$ and so, it is constant and by property (c) of Kempf-Ness function, $\exp(\R\xi) \subset G_y$, $\xi \in \liep_y.$
\end{proof}

Let $X^\beta := \{z\in X : \beta_X(z) = 0\}.$ $G^\beta$ preserves $X^\beta$ \cite[Prop. 2.9]{Stability} and $X^\beta$ is the disjoint union of closed submanifold of $X$ \cite{heinzner-schwarz-stoetzel}. The following result is proved in \cite[Proposition 2.10,p.92]{Stability}
\begin{proposition}\label{restric}
The restriction $(\mup)_{\vert_{X^\beta}}$ takes value on $\liep^\beta$ and so it coincides with the $G^\beta$-gradient map $(\mu_{\liep^\beta})_{\vert_{X^\beta}}$.
\end{proposition}
\begin{corollary}\label{cor1}
If $x\in X^\beta$ is $G^\beta$-polystable, then $x$ is $G$-polystable.
\end{corollary}

\begin{theorem}\label{main result}
A point $x\in X$ is polystable if and only if $\lambda_x \geq 0$ and for any $p\in Z(x)$ there exists $p'\in Z(x)$ such that $p$ and $p'$ are connected.
\end{theorem}

\begin{proof}
    Let $x\in X.$ If $Z(x) = \emptyset$, $\lambda_x > 0$ and by Theorem \ref{stable}, $x$ is stable and hence polystable. Suppose $Z(x) \neq \emptyset.$ Let $p\in Z(x).$ Let $\beta \in S(\liep)$ such that $p = e(\beta).$ suppose $p\in Z(x)$ is chosen such that the of the torus $T_\beta$ satisfies
$$
\dim T_\beta = \max_{\eta \in e^{-1}(Z(x))} \dim T_\eta.
$$

By assumption there is a geodesic $\alpha \in M$ such that $\alpha(+\infty) = p\in Z(x)$ and $\alpha(-\infty) = p' \in Z(x).$ Assume $\alpha(t) = g\exp(t\xi)K,$ $g\in G.$ Then $p = g\cdot e(\xi)$ and $p' = g\cdot e(-\xi).$ By Lemma \ref{limit}, $\xi \in \liep_y$ where $y = g^{-1}x.$ Moreover, since $e(\beta) = p = g\cdot e(\xi),$ using Lemma \ref{torus},
$$\dim T_\xi = \dim T_\beta = \max_{\eta \in e^{-1}(Z(x))} \dim T_\eta.
$$
Let $\mathfrak t_\xi$ be the Lie algebra of $T_\xi$. Then $\mathfrak a=i \mathfrak t_\xi \cap \liep^\xi$ is an Abelian subalgebra of $\liep^\xi$ different from zero since $\beta \in \lia$. Since $T_\xi =\overline{\exp(i\R \xi)}$ fixes $y$ it follows that $\lia \subseteq \lieg_y$

Let $Y$ be the connected component of $X^{\lia}$ containing $y$. By Lemma \ref{lemcomp}, $(G^\lia)^o=(K^\lia)^o \exp(\liep^\lia)$ is compatible and it preserves $Y$. By Proposition \ref{restric}  we get $(\mup)|_{Y}=\mu_{\liep^\lia}$. Hence, if $y$ is $(G^\lia)^o$-polystable, then it is $G$-polystable. We split $\mathfrak{p}^\lia = \text{span}(\lia) \oplus \mathfrak{p}^{'}$, where $\mathfrak{p}^{'}$ is the orthogonal of $\lia$ and so it is a $K^\lia$-invariant splitting. . 

Claim: $\lambda_y(e(\beta')) > 0$ for all $\beta' \in S(\liep').$ Indeed, we prove this claim by contradiction. Suppose there exists $\beta' \in S(\liep')$ such that $\lambda_y(e(\beta')) = 0.$ Hence $[\xi, \beta'] = 0$ by the choice of $\xi$ and $\beta',$  and they are linearly independent.
Let $a > 0.$ Since $[\xi,\beta']=0$ and $\xi \in \lieg_y$, by Lemma \ref{le1} it follows
$$
\lim_{t\to+\infty} \Phi(y,\exp(t(\xi+a\beta'))=\lim_{t\to+\infty} \Phi(y,\exp(t(\xi))+a \lim_{t\to+\infty} \Phi(y,\exp(t\beta')).
$$
Since $\lambda_y(e(\xi)) = \lambda_y((\beta')) = 0,$ by Lemma \ref{le3}, it follows that
$$
\lim_{t\to+\infty} \Phi(y,\exp(t(\xi+a\beta'))=0.
$$
Applying, again, Lemma \ref{le3}, we have
$$
\lambda_y(e(\frac{\xi + a\beta'}{\parallel \xi + a\beta'\parallel}))= 0,
$$
and so the vector $$\frac{\xi + a\beta'}{\parallel \xi + a\beta'\parallel}$$ belongs to $e^{-1}(Z(y)).$

We claim that for some $a>0$, $\dim T_{\xi+a\beta'}>\dim T_\xi$.

Let $T'= \overline{\exp(\R i\xi + \R i\beta')}\subseteq (U^\xi)^o$ and $T_{\beta'}=\overline{\exp(\R i \beta')}$.
 Let $U'\subseteq ( U^\xi)^o$ be a compact connected subgroup such that  the morphism 
\[
T_\xi \times U' \to (U^\xi)^o, (a,b) \mapsto (ab),
\]
is surjective with finite center. Since $\beta' \notin \lia$, it follows that $i\beta \notin \mathfrak t_\xi$. Hence, $T_{\beta'}\subseteq U'$ and the morphism
$$
f : T_\xi \times T_{\beta'} \to T', \quad f(a,b) = ab
$$ is a finite covering. Let $\{e_1, \cdots, e_n\},$ respectively $\{e^{'}_1, \cdots, e^{'}_m\}$, be a basis of the lattice $\ker \exp \subset \liet_\xi$, respectively $\ker \exp \subset \liet_{\beta'}.$ If $i\xi = X_1e_1 + \cdots + X_ne_n$ and $i\beta' = Y_1e'_1 + \cdots + Y_me'_m,$ then $i(\xi + a\beta' )= X_1e_1 + \cdots + X_ne_n + aY_1e'_1 + \cdots + aY_me'_m.$
Denote by $T'_{\xi +a \beta}$ the closure of $\exp(\R(i(\xi+a\beta'))$.
Since $f$ is a covering, $\dim T_{\xi + a\beta'} = \dim T'_{\xi + a\beta'}.$ Hence,
$$
\dim T_{\xi + a\beta'} = \dim_\Q (\Q X_1 + \cdots + \Q X_n + \Q aY_1 + \cdots + \Q aY_m),
$$
see for instance Dustermatt Kolk, Lie groups p. 61.
Since $\beta' \neq 0,$ $Y_j \neq 0$ for some $j.$ Choose $a$ such that $aY_j \notin \Q X_1 + \cdots + \Q X_n.$ Then $\dim T_{\frac{\xi + a\beta'}{\parallel \xi + a \beta \parallel}} > \dim T_\xi$ which is a contradiction. Therefore, $\lambda_y > 0$ on $e(S(\liep'))$. By Lemma \ref{le2}, $\Phi(y,\cdot)$ is linearly proper on $\liep'$. This implies that $\Phi(y,\cdot)$ is bounded from below on $\liep'$ and
$$
m=\mathrm{inf}_{\alpha \in \liep'} \Phi(y,\exp(\alpha)),
$$
is achieved. We claim that
$$
m=\mathrm{inf}_{\alpha \in \liep^\lia} \Phi(y,\exp(\alpha)).
$$
Indeed, let $v\in \liep^\lia$. Then $v=v_1 + v_2$, where $v_1 \in \lia$ and $v_2\in \liep'$. By the cocycle condition, keeping in mind that $[v_1,v_2]=0$ and $v_1 \in \lieg_y$, we get
$$
\Phi(y,\exp(v))=\Phi(y,\exp(v_1))+\Phi(y,\exp(v_2)).
$$
We claim that $\Phi(y,\exp(v_1))=0$. Indeed, since $\lia \subset \lieg_y$, If $w \in S(\lia)$ then by formula $(\ref{formula})$  we get
$$
\lambda_y (e(w))=\langle \mu_{\liep^\lia} (y), -w \rangle \geq 0,
$$
for any $w\in S(\lia)$. This implies $\lambda_y (e(-w))=-\lambda_y (e(w))$ and so $\lambda_y (e(w))=0$ for any $w\in \lia$.

Let $w\in \lia\setminus\{0\}$ and let $s: \R \lra \R$ be the function $s(t)=\Phi(y,\exp(tw))$. Since $\exp(tw)y=y$ for any $t\in \R$, it follows that $s(t)$ is a linear function. Therefore, $s(t)=b t$ for some $b\in \R$. On the other hand
\[
0=\lambda_y (e \big(\frac{w}{\parallel w \parallel} \big))=\lim_{t\to +\infty}\frac{d}{dt}\Phi(y, \exp(tw))=b.
\]
This proves
\[
\mathrm{inf}_{\alpha \in \liep'} \Phi(y,\exp(\alpha))=\mathrm{inf}_{\alpha \in \liep^\lia} \Phi(y,\exp(\alpha)),
\]
and so $\Phi_y : (G^\lia)^o /(K^\lia)^o \lra \R$ has a minumum and so a critical point. By Proposition \ref{polystable-I}, it follows that $y$ is $(G^\lia)^o$ polystable and by Corollary \ref{cor1}, $y$ is $G$-polystable.
Suppose $x$ is polystable. There exists $g\in G$ such that $\mu_\liep(gx) = 0.$ Let $y = gx$ and fix $\beta \in \liep.$ Since the Kempf-Ness function is convex along geodesics,
$$
\lambda_y(e(\beta)) = \lim_{t\to \infty}\frac{d}{dt}\phi_y(\exp(t\beta)K) = \lim_{t\to \infty}\frac{d}{dt}\phi_y(\exp(-t\beta))\geq \frac{d}{dt}|_{t= 0}\phi_y(\exp(-t\beta)) = \langle \mu_\liep(y), \beta\rangle = 0.
$$ This shows that $\lambda_y \geq 0$ on $\partial_\infty M.$ By (\ref{Ginvariant}), $\lambda_x \geq 0.$  By Lemma \ref{compatible-Stabilizer}, $G_y$ is compatible with $\lieg_y = \liek_y \oplus \liep_y$ and $Z(y) = e(S(\liep_y)).$ Suppose there exist $p = e(\beta) \in Z(y).$ Then $e(-\beta)\in Z(y)$ also. Furthermore, $e(\beta)$ and $e(-\beta)$ are connected by the geodesic $[\exp(t\beta)].$ This means that the condition of the Theorem holds for $Z(y).$ Now, for $p\in Z(x),$ $g\cdot p\in Z(y).$ Let $q\in Z(y)$ be connected to $g\cdot p$ by a geodesic $\alpha.$ Then the geodesic $g^{-1}\circ \alpha$ connects $p$ to $g^{-1}\cdot q\in Z(x).$ This concludes the proof of the theorem.

\end{proof}

\begin{corollary}
    $x\in X$ is polystable if and only if there exists $\beta \in S(\liep),$ $y\in G\cdot x$ and $g\in (G^\beta)^o$ such that $\lambda_x(e(\beta)) = 0$ and $\mu_\liep(gy) = 0.$
\end{corollary}

\end{document}